\documentclass{coulonpaper}

\def\coloneqq{\mathrel{\mathop\mathchar"303A}\mkern-1.2mu=}
\begin{document}

\title{Farell-Jones via Dehn fillings.}
\author{Yago Antol\'in, R\'emi Coulon, Giovanni Gandini}
\date{}
\maketitle

\begin{abstract}
Following the approach of Dahmani, Guirardel and Osin, we extend the  group theoretical Dehn filling theorem to  show that the pre-images of infinite order elements have a certain structure of a free product.
We then apply this result to show that groups hyperbolic relative to residually finite groups satisfying the Farrell-Jones conjecture, they satisfy the Farrell-Jones conjecture.
\end{abstract}

%
\section{Introduction}
%
\label{sec: intro}

The \emph{Farrell-Jones conjecture} aims to describe the algebraic $K$-theory of a group ring  in terms of the algebraic $K$-theories of simpler group rings. 
This is formulated by claiming that a certain map  is an isomorphism \cite{Farrell:1993ke}. 
Bartels, L\"uck and Reich introduced a series of isomorphism conjectures inspired by the insight of  Farrell and Jones, that are more general and have stronger implications and inheritance properties.    
Since many geometric obstructions live in the algebraic $K$-theory of the group ring, the validity of the Farrell-Jones conjecture, has several outstanding consequences, for example  the  \emph{Novikov conjecture}, the \emph{Borel conjecture}, the \emph{Kaplansky conjecture}, and the \emph{Serre conjecture} \cite{Luck:2005cf}.  
In the last 10 years that has been a tremendous amount of work devoted in solving the Farrell-Jones conjecture for several classes of groups, and often its solution gave the first solution to some of the famous conjectures above \cite{Luck:2005cf,Bartels:2008bn,Bartels:2012ca,Wegner:2012hh,Wegner:2013te,Bartels:2014jv}.

\medskip
Now we start by describing the statement of the version of the Farrell-Jones conjecture that we consider here. 
For a  group $G$, let   $\doubleunderline EG$ be  the classifying space of $G$ for the family of virtually cyclic subgroups. 
We recall that a  classifying space of $G$ for the family of virtually cyclic subgroups  is a $G$-CW-complex $X$ with the property that the fixed point subcomplex $X^H$ is contractible for every virtually cyclic subgroup $H$ of $G$ and empty otherwise. 
This space is generally a rather unusual $G$-space,  some  examples can be found in \cite{Luck:2005hf}.  
Now, let $\mathcal{A}$ be an additive $G$-category, the  projection $\doubleunderline EG \to \mbox{pt}$ induces an ``assembly map''
\begin{displaymath}
	\mbox{asmb}_n^{G, \mathcal{A}} : H_n^G(\doubleunderline EG; \mathbf S)\to  H_n^G(\mbox{pt}; \mathbf S),
\end{displaymath} 
where $H_n^G(\doubleunderline EG; \mathbf S)$ is a $G$-equivariant homology theory, and $\mathbf S$ can be either the Or$G$-spectrum $\bold{K_{\mathcal{A}} }$ or the Or$G$-spectrum $\bold{L_{\mathcal{A}}}$, please consult \cite{Bartels:2007hi} for various definitions and details. 
Here we say that a group  $G$ satisfies the \emph{Full Farrell-Jones Conjecture} if for any finite group $F$ the assembly map for the wreath product $G\wr F$ is an isomorphism in any additive $G$-category. 
When $\mathcal{A}$ is the category of finitely generated free modules over some ring $R$ with trivial $G$-action, then this reduces to say that the map $H_n^G(\doubleunderline EG;  \bold{K_{R}})\to  K_n(RG)$, is an isomorphism. 
In other words,  the algebraic $K$-theory $K_n(RG)$ of the group ring $RG$  can be constructed from the algebraic $K$-theories  $K_n(RH)$'s where $H$ ranges over the family of virtually cyclic subgroups.


\medskip
Recently, Bartels, using a new notion of ``coarse flow space''  solved the Farrell-Jones conjecture for all hyperbolic groups relative to a family of subgroups satisfying the Farrell-Jones conjecture \cite{Bartels:2015to}.  
Bartels' proof  uses a far reaching generalisation of the \emph{geodesic flow method} introduced in \cite{Farrell:1993ke} which was also adapted in the solution of the Farrell-Jones conjecture for hyperbolic groups \cite{Bartels:2008bn}. 
We extend the Dehn filling Theorem  following Dahmani, Guirardel, and Osin \cite{Dahmani:2011vu} and combine it with the solution of the Farrell-Jones conjecture for hyperbolic groups to give an  alternative proof of  the conjecture for relative hyperbolic groups when the peripheral subgroups  are residually finite in addition to satisfy  the Farrell-Jones conjecture. 
Note that this case covers the geometrically relevant examples of relatively hyperbolic groups, such as  fundamental groups of  complete hyperbolic manifolds of finite volume  and fundamental groups  of complete Riemannian manifolds  of finite volume  with pinched negative sectional curvature. 
Our proof relies on  the strong inheritance properties of the (Full) Farrell-Jones conjecture, L\"uck-Bartels-Reich solution for hyperbolic groups, and a detailed study of the group structure of relatively hyperbolic groups.
The conjecture remains unknown for many groups connected to this work, such as the outer automorphism groups of a right-angled Artin group  or even for $Out(F_n)$, mapping class groups and more generally for \emph{acylindrically hyperbolic groups}.\\

\medskip
Dehn fillings or Dehn surgery is a powerful tool to produce quotient groups and spaces of negative curvature. Group theoretical Dehn fillings were inspired by Thurston's hyperbolic Dehn surgery Theorem \cite{Thurston:1982ia} which says that if $M$ is an hyperbolic 3-manifold with a single torus cusp $C$, then for all but finitely
many $g\in \pi_1(C)$ the quotient group $\pi_1(M)/\normal g$ is infinite, non-elementary and word hyperbolic. 
The group theoretical version of this theorem for relatively hyperbolic groups is due to Groves and Manning \cite{Groves:2008ip} and Osin \cite{Osin:2007ia}. 
The theorem has been further generalized to the context of acylindrically hyperbolic groups by Dahamani, Guirardel and Osin \cite{Dahmani:2011vu}.  
In cite \cite{Dahmani:2011vu}, using Gromov's rotating families \cite{Gro01a} and windmills, the authors are able to describe the kernel of a Dehn filling and they show that it isomorphic to a free product.

\medskip
The strategy we follow to prove the Farrell-Jones conjecture is based on the stability of the conjecture under certain type of group extensions. 
More concretely, if $G$ is an extension of groups satisfying the Farrell-Jones conjecture, then $G$ itself satisfies the Farrell-Jones conjecture provided that the preimage in $G$ of any infinite cyclic subgroup of the quotient  satisfies the conjecture. 
Then, if one starts with a relatively hyperbolic group $G$ with residually finite parabolic subgroups satisfying the Farrell-Jones conjecture, one can use the Dehn fillings theorem  to obtain a short exact sequence $K\to G \stackrel{\pi}{\to} Q$, where $K$ and $Q$ satisfy the Farrell-Jones conjecture and the problem relies on understanding $\pi^{-1}(\langle q\rangle)$ for $q\in Q$ of infinite order.
Our main technical contribution is item (iii) of the theorem below.

\begin{theo}
	\label{res: Dehn filling with a reduced element intro}
	Let $G$ be finitely generated group hyperbolic relatively to a family of subgroups $\{P_1,\dots, P_n\}$.
	There is a finite set $\Phi \subseteq G\setminus\{1\}$ such that
	whenever we take finite index normal subgroups $N_i\trianglelefteq P_i$  with $N_i\cap \Phi=\emptyset$ for $i=1,\dots, n$ then the following hold:
	\begin{enumerate}
		\item $\bar G \coloneqq G/K$ is an hyperbolic group where $K$ is the normal subgroup of $G$ generated by $N_1\cup \dots \cup N_n$.
		\item there exists subsets $T_i$ of $G$ for $i=1,\dots, n$ such that $K$ is isomorphic to
		$*_{i=1}^n\left(*_{t\in T_i} N_i^t\right)$.
		\item for every $\bar g\in \bar G$ of infinite order, there is a pre-image $g$ of $\bar g$ under the natural map $G\to \bar G$ and  subsets $T_i'$ of $G$ for $i=1,\dots, n$ such that
		\begin{displaymath}
			\langle g, K \rangle = \langle g \rangle * \left[*_{i=1}^n \left(*_{t\in T_i'}N_i^t\right) \right].
		\end{displaymath}
	\end{enumerate}
\end{theo}

Note that (i) already appears in \cite{Osin:2007ia} and (ii) appears in \cite{Dahmani:2011vu}.
Our proof of (iii) follows the strategy of the one of (ii) of Dahmani, Guirardel and Osin, and for that, we introduce a variation of the windmills used in \cite{Dahmani:2011vu}.

\medskip
The group theoretical Dehn fillings Theorem (and variations) have proved to be a extremely useful tool in modern  geometric group theory.
The is a good number of interesting applications: it has been used to construct simple groups with arbitrarily large $\ell^2$-Betti number \cite{Osin:2013hr}, to prove that normal automorphisms of acylindrically hyperbolic groups with trivial finite radical are  inner \cite{Antolin:2013ur}, and it plays an important  role in the solution of the Virtual Haken conjecture \cite{Agol:2013wr}. 
Therefore, \autoref{res: Dehn filling with a reduced element intro}~\ref{enu: Dehn filling with a reduced element - free primages} is interesting not only for its application to the Farrell-Jones conjecture, but also for obtaining a better understanding of the Dehn fillings theorem itself.

\medskip
The structure of the paper is as follows: in \autoref{sec: hyperbolic geometry} we set notation and review all the basic definitions about hyperbolic geometry.  \autoref{sec: rotations families} is the core of the paper, we introduce the extended windmills and we prove \autoref{res: rotation family with a reduced element} which is a version of \autoref{res: Dehn filling with a reduced element intro} for groups acting on hyperbolic spaces. In \autoref{sec : relative hyperbolicity}, we collect the needed references to deduce \autoref{res: Dehn filling with a reduced element intro} from the results of \autoref{sec: rotations families} and finally in \autoref{sec : farrell jones} we get the main application of the paper, namely \autoref{res : farrell jones for relatively hyperbolic} where we get our version of the  Farrell-Jones conjecture for relatively hyperbolic groups.


\medskip
 
\paragraph{Acknowledgments.}
As we start thinking about this question Bartels' general solution had not appeared.
The authors are thankful to D.~Osin and A.~Bartels who encouraged us though to write down this alternative approach.
The authors would like to thank the organizers of the Ventotene International Workshops (2015) were part of this paper was written. 
The first author acknowledge partial support from the
Spanish Government through grant number MTM2014-54896-P.
The third author was supported by the Danish
National Research Foundation
(DNRF) through the Centre for Symmetry and Deformation.

%
\section{Hyperbolic geometry}
%
\label{sec: hyperbolic geometry}

\paragraph{Notations and vocabulary.}
Let $X$ be a metric length space.
Given two points $x$ and $x'$ of $X$, we denote by $\dist[X]x{x'}$ (or simply $\dist x{x'}$) the distance between them.
Let $Y$ be a subset of $X$.
We write $d(x,Y)$ for the distance between a point $x \in X$ and $Y$.
We write $B(x,r)$ for the closed ball of center $x $ and radius $r$.


\paragraph{The four point inequality.}

The Gromov product of three points $x,y,z \in X$  is defined by 
\begin{displaymath}
	\gro xyz = \frac 12 \left\{  \dist xz + \dist yz - \dist xy \right\}.
\end{displaymath}
For the remainder of this section, we assume that the space $X$ is \emph{$\delta$-hyperbolic}, i.e. for every $x,y,z,t \in X$,
\begin{equation}
\label{eqn: hyperbolicity condition 1}
	\gro xzt \geq \min\left\{ \gro xyt, \gro yzt \right\} - \delta,
\end{equation}
or equivalently
\begin{equation}
\label{eqn: hyperbolicity condition 2}
	\dist xz + \dist yt \leq \max\left\{  \dist xy + \dist zt, \dist xt + \dist yz  \right\} +2\delta.
\end{equation}

\rem If $X$ is 0-hyperbolic, then it can be isometrically embedded in an $\R$-tree, \cite[Chapitre 2, Proposition 6]{GhyHar90}. For our purpose though,
we will always assume that the hyperbolicity constant $\delta$ is positive. 
Indeed, every $0$-hyperbolic space is $\delta$-hyperbolic for every $\delta\geq 0$. 

\medskip

It is known that triangles in a geodesic hyperbolic space are  $4\delta$-thin (every side lies in the $4\delta$-neighborhood of the union of the two other ones). 
This can be stated through the following metric inequality.
In this statement the Gromov product $\gro xzs$ should be thought as a very small quantity.	
For every $x,y,z,s \in X$, 
\begin{equation}
\label{res: metric inequalities - thin triangle}
	\gro xys \leq \max \left\{ \dist xs - \gro yzx , \gro xzs  \right\} + \delta.
\end{equation}

\paragraph{The boundary at infinity.} 

Let $x$ be a base point of $X$.
A sequence $(y_n)$ of points of $X$ \emph{converges to infinity} if $\gro {y_n}{y_m}x$ tends to infinity as $n$ and $m$ approach to infinity.
The set $\mathcal S$ of such sequences is endowed with a binary relation defined as follows.
Two sequences $(y_n)$ and $(z_n)$ are related if 
\begin{displaymath}
	\lim_{n \rightarrow + \infty} \gro {y_n}{z_n}x = + \infty.
\end{displaymath}
If follows from (\ref{eqn: hyperbolicity condition 1}) that this relation is actually an equivalence relation.
The \emph{boundary at infinity} of $X$, denoted by $\partial X$, is the quotient of $\mathcal S$ by this relation.
If the sequence $(y_n)$ is an element in the class of $\xi \in \partial X$, we say that $(y_n)$  \emph{converges} to $\xi$ and  write
\begin{displaymath}
	\lim_{n \rightarrow + \infty} y_n = \xi.
\end{displaymath}
Note that the definition of $\partial X$ does not depend on the base point $x$.

\paragraph{Quasi-geodesics.}
In this article, unless otherwise stated a path is always a rectifiable path parametrized by arc length.

\begin{defi}
\label{def: quasi-geodesic}
	Let $l \geq 0$, $k \geq 1$ and $L\geq 0$.
	Let $f\colon X_1 \rightarrow X_2$ be a map between two metric spaces $X_1$ and $X_2$.
	We say that $f$ is a \emph{$(k,l)$-quasi-isometric embedding} if for every $x,x' \in X_1$, 
	\begin{equation*}
		k^{-1}\dist  {f(x)}{f(x')}  - l \leq \dist x{x'}\leq k\dist  {f(x)}{f(x')}   + l.
	\end{equation*}
	We say that $f$ is an \emph{$L$-local $(k,l)$-quasi-isometric embedding} if its restriction to any subset of diameter at most $L$ is a $(k,l)$-quasi-isometric embedding.
	Let $I$ be an interval of $\R$.
	A path $\gamma\colon I \rightarrow X$ that is a $(k,l)$-quasi-isometric embedding is called a \emph{$(k,l)$-quasi-geodesic}.
	Similarly, we define \emph{$L$-local $(k,l)$-quasi-geodesics}.
\end{defi}

\rems 
We assumed that our paths are rectifiable and parametrized by arc length.
Thus a $(k,l)$-quasi-geodesic $\gamma\colon I \rightarrow X$ satisfies a more accurate property: for every $t,t' \in I$,
\begin{equation*}
	\dist{\gamma(t)}{\gamma(t')} \leq \dist t{t'} \leq k\dist{\gamma(t)}{\gamma(t')} +l.
\end{equation*}
In particular, if $\gamma$ is a $(1,l)$-quasi-geodesic, then for every $t,t',s \in I$ with $t\leq s\leq t'$, we have $\gro {\gamma(t)}{\gamma(t')}{\gamma(s)}\leq l/2$.
Since $X$ is a length space for every $x,x' \in X$, for every $l>0$, there exists a $(1,l)$-quasi-geodesic joining $x$ and $x'$.

\paragraph{}Let $\gamma\colon \R_+ \rightarrow X$ be a $(k,l)$-quasi-geodesic.
There exists a point $\xi \in \partial X$ such that for every sequence $(t_n)$ diverging to infinity, $\lim_{n \rightarrow + \infty}\gamma(t_n) = \xi$.
In this situation we consider $\xi$ as an endpoint (at infinity) of $\gamma$ and write $\lim_{t \rightarrow + \infty} \gamma(t) = \xi$.
In this article we are mostly using $L$-local $(1,l)$-quasi-geodesics.
Thus we state the stability of quasi-geodesics for this kind of paths.

\begin{coro}{\rm \cite[Corollaries 2.6 and 2.7]{Coulon:2014fr}} \quad
\label{res: stability (1,l)-quasi-geodesic}
	Let $l_0 \geq 0$.
	There exists $L=L(l_0,\delta)$ which only depends on $\delta$ and $l_0$ such that
	for every $l \in \intval 0{l_0}$,
	and every $L$-local $(1,l)$-quasi-geodesic $\gamma \colon I \rightarrow X$, the following hold:
	\begin{enumerate}
		\item the path $\gamma$ is a (global) $(2,l)$-quasi-geodesic,
		\item for every $t,t',s \in I$ with $t \leq s \leq t'$, we have $\gro{\gamma(t)}{\gamma(t')}{\gamma(s)} \leq l/2 + 5 \delta$,
		\item for every $x \in X$, for every $y,y'$ lying on $\gamma$, we have $d(x,\gamma) \leq \gro y{y'}x + l + 8 \delta$.
		\item the Hausdorff distance between $\gamma$ and any other $L$-local $(1,l)$-quasi-geodesic joining the same endpoints (possibly in $\partial X$) is at most $2l+5\delta$.
	\end{enumerate}
\end{coro}
	
\rem Using a rescaling argument, one can see that the best	
value for the parameter $L=L(l,\delta)$ satisfies the following property: for all $l,\delta \geq 0$ and $\lambda >0$, $L(\lambda l, \lambda \delta) = \lambda L(l,\delta)$.
This allows us to define a parameter $L_S$ that will be use all the way through.

\begin{defi}
	\label{def: constant LS}
	Let $L(l,\delta)$ be the infimum value for the parameter $L$ given in \autoref{res: stability (1,l)-quasi-geodesic}.
	We denote by $L_S$  a number larger than $500$ such that $L(10^5\delta,\delta) < L_S\delta$. 
\end{defi}

The stability of quasi-geodesics (\autoref{res: stability (1,l)-quasi-geodesic}) has a discrete analogue that we state below.

\begin{prop}[Stability of discrete quasi-geodesics]{\rm \cite[Proposition~2.9]{Coulon:2014fr}} \quad
\label{res: stability discrete quasi-geodesic}
	Let $l > 0$.
	There exists $L = L(l,\delta)$ which only depends on $\delta$ and $l$ such that
	for every sequence of points  $x_0, \dots, x_m$ in $X$, satisfying that
	\begin{enumerate}
		\item for every $i \in \intvald 1{m-1}$, $\gro {x_{i-1}}{x_{i+1}}{x_i} \leq l$,
		\item for every $i \in \intvald 1{m-2}$, $\dist {x_{i+1}}{x_i} \geq L$.
	\end{enumerate}
	Then for all $i \in \intvald 0m$, the inequality  $\gro {x_0}{x_m}{x_i} \leq l+ 5\delta$ holds.
	Moreover, for all $p \in X$ there exists $i \in \intvald 0{m-1}$ such that $\gro{x_{i+1}}{x_i}p \leq \gro {x_0}{x_m}p + 2l + 8\delta$.
\end{prop}

\paragraph{Quasi-convex subsets.}

Let $Y$ be a subset of $X$.
Let $\alpha \geq 0$.
We denote by $Y^{+\alpha}$, the \emph{$\alpha$-neighborhood} of $Y$, i.e. the set of points $x \in X$ such that $d(x, Y) \leq \alpha$.
Let $\eta \geq 0$.
A point $p$ of $Y$ is an \emph{$\eta$-projection} of $x \in X$ on $Y$ if $\dist xp \leq d(x,Y) +\eta$.
A 0-projection is simply called a \emph{projection}.

\begin{defi}
\label{def: quasi-convex}
	Let $\alpha \geq 0$.
	A subset $Y$ of $X$ is \emph{$\alpha$-quasi-convex} if for every $x \in X$ and for every $y,y' \in Y$ the inequality $d(x,Y) \leq \gro y{y'}x + \alpha$ holds.
\end{defi}

\begin{lemm}[Projection on a quasi-convex] {\rm \cite[Chapitre 10, Proposition 2.1]{CooDelPap90}} \quad
\label{res: proj quasi-convex}
	Let $Y$ be an $\alpha$-quasi-convex subset of $X$. 
	Let $x, x' \in X$.
	\begin{enumerate}
		\item \label{enu: proj quasi-convex - gromov product}
		If $p$ is an $\eta$-projection of $x$ on $Y$, then for all $y \in Y$, $\gro xyp \leq \alpha + \eta$.
		\item \label{enu: proj quasi-convex - distance two points }
		If $p$ and $p'$ are respective $\eta$- and $\eta'$-projections of $x$ and $x'$ on $Y$, then 
		\begin{displaymath}
			\dist p{p'} \leq \max \left\{\fantomB \dist x{x'}-\dist xp - \dist {x'}{p'} +2\epsilon, \epsilon \right\},
		\end{displaymath}
		where $\epsilon = 2 \alpha + \eta + \eta' + \delta$.
	\end{enumerate}
\end{lemm}

\begin{lemm}[Neighborhood of a quasi-convex]  {\rm \cite[Chapitre 10, Proposition 1.2]{CooDelPap90}} \quad
\label{res: neighborhood of a quasi-convex}
	Let $\alpha \geq 0$.
	Let $Y$ be an $\alpha$-quasi-convex subset of $X$.
	For every $A \geq \alpha$, the $A$-neighborhood of $Y$ is $2 \delta$-quasi-convex.
\end{lemm} 
	
\begin{defi}
\label{def: hull}
	Let $Y$ be a subset of $X$.
	The \emph{hull} of $Y$ denoted by $\hull Y$ is the union of all $(1, \delta)$-quasi-geodesics joining two points of $Y$.
\end{defi}

\begin{lemm}{\rm \cite[Lemma~2.18]{Coulon:2014fr}}
\label{res: hull quasi-convex}
	The hull of any subset of $X$ is $6\delta$-quasi-convex.
\end{lemm}

\begin{lemm}{\rm \cite[Lemma~2.19]{Coulon:2014fr}}
\label{res: gromov product and hull}
	Let $Y$ and $Z$ be two subsets of $X$.
	Let $x$ be a point of $X$.
	Assume that for all $y \in Y$ and for all $z \in Z$, the inequality $\gro yzx \leq \alpha$ holds.
	Then for all $y \in \hull Y $ and for all $z \in \hull Z$, we have that $\gro yzx \leq \alpha + 3\delta$.
\end{lemm}

\paragraph{Isometries of a hyperbolic space.}

Let $x$ be a point of $X$.
An isometry $g$ of $X$  is either
\begin{itemize}
	\item \emph{elliptic}, i.e. the orbit $\langle g \rangle x$ is bounded,
	\item \emph{loxodromic}, i.e. the map from $\Z$ to $X$ that sends $m$ to $g^m x$ is a quasi-isometric embedding,
	\item or \emph{parabolic}, i.e. it is neither loxodromic or elliptic.		
\end{itemize}
Note that these definitions do not depend on the point $x$.
In order to measure the action of $g$ on $X$, we use two translation lengths.
By the \emph{translation length} $\len[espace=X]g$ (or simply $\len g$) we mean 
\begin{displaymath}
	\len[espace=X] g \coloneqq \inf_{x \in X} \dist {gx}x.
\end{displaymath}
The \emph{asymptotic translation length} $\len[stable, espace=X] g$ (or simply $\len[stable]g$) is
\begin{displaymath}
	\len[espace=X,stable] g \coloneqq \lim_{n \rightarrow + \infty} \frac 1n \dist{g^nx}x.
\end{displaymath}
These two lengths are related as follows.

\begin{prop}{\rm \cite[Chapitre 10, Proposition 6.4]{CooDelPap90}}\quad
\label{res: translation lengths}
	Let $g$ be an isometry of $X$. 
	Its translation lengths satisfy
	\begin{equation*}
		\len[stable]g \leq \len g \leq \len[stable] g + 16\delta.
	\end{equation*}
\end{prop}
The isometry $g$ is loxodromic if and only if its asymptotic translation length is positive \cite[Chapitre 10, Proposition 6.3]{CooDelPap90}.
In this case $g$ fixes exactly two points of $\partial X$ \cite[Chapitre 10, Proposition 6.6]{CooDelPap90} which are
\begin{displaymath}
	g^-  \coloneqq \lim_{n \rightarrow - \infty} g^nx \text{ and } g^+  \coloneqq \lim_{n \rightarrow + \infty} g^nx.
\end{displaymath}

Recall that $L_S$ is the parameter given by the stability of quasi-geodesics (see \autoref{def: constant LS}).

\begin{defi}
\label{def: cylinder loxodromic element}
	Let $g$ be a loxodromic isometry of $X$.
	We denote by $\Gamma_g$ the union of all $L_S\delta$-local $(1,\delta)$-quasi-geodesics joining $g^-$ to $g^+$.
	The \emph{cylinder} of $g$, denoted by $Y_g$, is the $20\delta$-neighborhood of $\Gamma_g$.
\end{defi}

\rem Note that if $g$ is a loxodromic isometry of $X$,
	then both $\Gamma_g$ and $Y_g$ are invariant under the 
	$\langle g\rangle$-action.

\begin{lemm}{\rm \cite[Lemma~2.31]{Coulon:2014fr}} \quad
\label{res: cylinder strongly quasi-convex}
	Let $g$ be a loxodromic isometry of $X$.
	The cylinder of $g$ is $2\delta$-quasi-convex.
\end{lemm}

\begin{lemm}
\label{res: moving point of a cylinder}
	Let $g$ be a loxodromic isometry of $X$.
	For every $x \in X$, $\dist {gx}x  \leq \len g + 2 d(x, Y_g) + 112\delta$.
\end{lemm}

\begin{proof}
	Let us denote by $A_g$ the set of points $z \in X$ such that $\dist{gz}z < \len g + 8\delta$.
	It is known that $Y_g$ lies in the $52 \delta$-neighborhood of $A_g$ \cite[Lemma~2.32]{Coulon:2014fr},
	implying that for $y\in Y_g$, $|y-gy|\leq \len g + 112\delta$.
	Let $x$ be a point of $X$ and $y$ a $\eta$-projection of $x$ on $Y_g$.
	It follows that 
	\begin{align*}
		\dist{gx}x &\leq \dist {gx}{gy} + \dist {gy}{y} + \dist yx \\
		&\leq d(gx, Y_g)+\eta+ \dist{gy}y +   d(x,Y_g)  +  \eta \leq \len g   + 112 \delta + 2 d(x,Y_g) +2 \eta.
	\end{align*}
	The last inequality holds for every $\eta >0$ which completes the proof.
\end{proof}

\begin{defi}
\label{def: nerve}
	Let $g$ be an isometry of $X$. 
	Let $l \geq 0$.
	A path $\gamma \colon \R \rightarrow X$ is called an \emph{$l$-nerve} of $g$ if there exists $T \in \R$ with $\len g \leq T \leq \len g +l$ such that $\gamma$ is a $T$-local $(1, l)$-quasi-geodesic and for every $t \in \R$, $\gamma(t+T) = g\gamma(t)$.
	The parameter $T$ is called the \emph{fundamental length} of $\gamma$.
\end{defi}

\rem
For every $l >0$, one can construct an $l$-nerve of $g$ as follows.
Let $\eta >0$.
There exists $x \in X$ such that $\dist {gx}x < \len g + \eta$.
Let $\gamma\colon \intval 0T \rightarrow X$ be a $(1, \eta)$-quasi-geodesic joining $x$ to $gx$.
In particular $\len g \leq T < \len g + 2\eta$.
We extend $\gamma$ into a path $\gamma \colon \R \rightarrow X$ in the following way: for every $t \in [0,T)$, for every $m \in \Z$, $\gamma(t+mT) = g^m\gamma(t)$.
It turns out that $\gamma$ is a $T$-local $(1,2\eta)$-quasi-geodesic.
Thus if $\eta$ is chosen sufficiently small then $\gamma$ is an $l$-nerve.

\paragraph{}This kind of path will be used to simplify some proofs.
Let $\gamma$ be a $\delta$-nerve of $g$.
If $\len g > L_S\delta$ (in particular $g$ is loxodromic) then $\gamma$ is contained in $\Gamma_g\subseteq Y_g$.
By stability of quasi-geodesics (Corollary \ref{res: stability (1,l)-quasi-geodesic}~(iii)) $\gamma$ is actually $9\delta$-quasi-convex.
Moreover it joins $g^-$ to $g^+$.
By Corollary \ref{res: stability (1,l)-quasi-geodesic}~(iv), any other $(1,\delta)$-quasi-geodesic $\gamma'$ joining
$g^-$ and $g^+$ is at Hausdorff distance at most $7\delta$ of $\gamma$.
Thus $Y_g$ lies in the $27\delta$-neighborhood of $\gamma$.
Hence it provides a $g$-invariant line than can advantageously be used as a substitution for a cylinder.

%
\section{Rotation families}
%
\label{sec: rotations families}

\paragraph{Original settings.} 
In this section we extend the framework of rotation families given by F.~Dahmani, V.~Guirardel and D.~Osin in \cite{Dahmani:2011vu}.
Let $G$ be a group acting by isometries on a $\delta$-hyperbolic length space $X$.

\begin{defi}
\label{def: rotation family}
	Let $\sigma >0$.
	A \emph{$\sigma$-rotation family} is a non-empty collection $\mathcal R$ of pairs $(H,v)$ where $H$ is a subgroup of $G$  and $v$ a point of $X$ satisfying the following properties.
	\begin{labelledenu}[R]
		\item \label{enu: rotation family - large angle}
		For every $(H,v) \in \mathcal R$,  for every $x \in B(v, \sigma/10)$, and for every $h \in H\setminus\{1\}$, the equality $\dist{hx}x = 2\dist vx$ holds.
		\item \label{enu: rotation family - apices apart}
		For every $(H,v), (H',v') \in \mathcal R$, if $(H,v) \neq (H',v')$, then $\dist v{v'} \geq \sigma$.
		\item \label{enu: rotation family - G invariant}
		 For all $g \in G$ and for all $(H,v) \in \mathcal R$, the pair $ (gHg^{-1}, gv)$ belongs to $\mathcal R$. In particular, $\mathcal R$ has a natural structure of $G$-set.
	\end{labelledenu}
\end{defi}

\rem It follows from~\ref{enu: rotation family - apices apart} and \ref{enu: rotation family - G invariant} that for every $(H,v) \in \mathcal R$, $H$ is actually a normal subgroup of $\stab v$.

\notas Let $(H,v) \in \mathcal R$. The idea is that each element  $h \in H$ acts on $X$ like a rotation of center $v$ and very large angle - see Axiom~\ref{enu: rotation family - large angle}.
Therefore $v$ is called an \emph{apex} and $H$ a \emph{rotation group}.
If $\mathcal S$ is a subset of $\mathcal R$ denote by $v(\mathcal S)$ the set of all apices $v$ such that $(H,v) \in \mathcal S$.
Similarly $H(\mathcal S)$ stands for the set of all rotation groups $H$ with $(H,v) \in \mathcal S$.
Given a subset $Y$ of $X$, we denote by $K_Y$ the subgroup of $G$ generated by all the rotation groups $H$ where $(H,v) \in \mathcal R$ and $v \in Y$.
The (normal) subgroup generated by all the rotation groups is simply denoted by $K$.

\paragraph{}In their work \cite{Dahmani:2011vu}, F.~Dahmani, V.~Guirardel and D.~Osin  use the properties of such a family to study the structure of $K$ and the quotient $\bar G = G/K$ .
Among other things, they prove the following facts.
See also \cite{Coulon:2014fr} for a slightly different exposition of the last two points.

\begin{theo}
\label{res: standard rotation family}
	There exists $\sigma_0> 0$ which only depends on $\delta$, such that for every  $\sigma \geq \sigma_0$, 
	and every $\sigma$-rotating family $\mathcal{R}$,  the following holds.
	\begin{enumerate}
		\item There exists a subset $\mathcal S$ of $\mathcal R$ such that $K$ is isomorphic to the free product of the element of $H(\mathcal S)$.
		\item The subgroup $K$ acts properly on $X\setminus v(\mathcal R)$.
		\item The quotient $\bar X  = X / K$ is a $\bar \delta$-hyperbolic length space with $\bar \delta \leq 900 \delta$.
	\end{enumerate}
\end{theo}
For a proof of the \autoref{res: standard rotation family}, see \cite[Theorem 5.3]{Dahmani:2011vu} for (i) (or \autoref{res: advanced rotation family} below),  \cite[Corollary 3.12]{Coulon:2014fr} for (ii) and \cite[Propositions 3.14 and 3.18]{Coulon:2014fr} for (iii) (or, \cite[Proposition 5.28]{Dahmani:2011vu} for (ii) and (iii)).
\paragraph{Extended windmill.}
The goal of this section is to improve the approach of F.~Dahmani, V.~Guirardel and D.~Osin in order to study the structure of the subgroup of $G$ generated by $K$ and some subgroup of $G$.
More precisely we prove the following statement.

\begin{theo}
\label{res: advanced rotation family}
	There exists $\sigma_0 >0$ which only depends on $\delta$ such that the following holds.
	Assume that $\sigma \geq \sigma_0$.
	Let $\mathcal R$ be an $\alpha$-rotating family.	
	Let $Y$ be a $2\delta$-quasi-convex subset of $X$ and $N$ a subgroup of $G$ stabilizing $Y$ with the following properties
	\begin{enumerate}
		\item For every $(H,v) \in \mathcal R$, for every $h \in H \setminus \{1\}$, for every $y,y' \in Y$, $\gro y{hy'}v \leq 100\delta$.
		\item For every $(H,v) \in \mathcal R$, $\stab v \cap N = \{1\}$.
	\end{enumerate}
	Then there exists a subset $\mathcal S$ of $\mathcal R$ such that the subgroup generated by $N$ and $K$ is isomorphic to the free product of $N$ and the elements of $H(\mathcal S)$.
\end{theo}

If $N$ is trivial then we recover the first point of \autoref{res: standard rotation family}.
The rest of this section is dedicated to the proof of the theorem.
For that, we extend the notion of windmill (see \autoref{def: extended windmill}).
But first we need to define $\sigma_0$.
Applying to \autoref{res: stability discrete quasi-geodesic} with $l=105\delta$ there exists $\sigma_0=L(105\delta,\delta)$ such that for any sequence of points  $y_0, \dots, y_{m+1}$ in $X$, satisfying that 
\begin{enumerate}
	\item for every $i \in \intvald 1m$, $\gro{y_{i+1}}{y_{i-1}}{y_i} \leq 105\delta$,
		\item for every $i \in \intvald 1{m-1}$, $\dist{y_{i+1}}{y_i} \geq \sigma_0$,
\end{enumerate}
then, the inequality $\gro {y_0}{y_{m+1}}{y_i} \leq 110\delta$ holds for for every $i \in \intvald 0{m+1}$.
Moreover, for all $x \in X$, there exists $i \in \intvald 0m$, such that $\gro{y_{i+1}}{y_i}x \leq \gro {y_0}{y_{m+1}}x + 218\delta$.
Without loss of generality we can assume that $\sigma_0$ is greater than $10^{10}\delta$.

\medskip
From now on we assume that $\mathcal R$ is a $\sigma$-rotation family with $\sigma \geq \sigma_0$.
Let us recall now some basic facts.

\begin{lemm}{\rm \cite[Lemma~3.3]{Coulon:2014fr}}
\label{res: rotation family - small product at the apex}
	Let $(H,v) \in \mathcal R$.
	Let $h \in H\setminus\{1\}$.
	For every $x \in X$, $\gro x{hx}v \leq  2\delta$.
\end{lemm}

\begin{lemm}
\label{res: rotation family - small product and far qconvex}
	Let $(H,v) \in \mathcal R$.
	Let $h \in H \setminus \{1\}$.
	Let $Y$ be an $\alpha$-quasi-convex subset of $X$ such that $d(v,Y) > \alpha + 3\delta$.
	For every $y,y' \in Y$ we have $\gro y{hy'}v \leq 3\delta$.
\end{lemm}

\begin{proof}	Let $y, y' \in Y$.
	Combining the four point inequality (\ref{eqn: hyperbolicity condition 1}) with \autoref{res: rotation family - small product at the apex} we get
	\begin{equation*}
		\min\left\{ \gro y{y'}v,  \gro y{hy'}v \right\} \leq \gro {y'}{hy'}v + \delta \leq 3\delta.
	\end{equation*}
	Recall that since $Y$ is $\alpha$-quasi-convex (Definition \ref{def: quasi-convex}),
	we  have $\gro y{y'}v \geq d(v,Y) - \alpha > 3 \delta$.
	Consequently, the minimum cannot be achieved by $ \gro y{y'}v$, and hence $\gro y{hy'}v \leq 3\delta$.
\end{proof}


\begin{defi}
\label{def: extended windmill}
	Let $W$ be a subset of $X$, $N$ a subgroup of $G$ and $V$ a subset of $v(\mathcal R)$.
	Let $L$ be the subgroup of $G$ generated by $N$ and $K_V$.
	The triple $(W,N,V)$ is an \emph{extended windmill} if the following holds.
	\begin{labelledenu}[W]
		\item \label{enu: extended windmill - quasi-convex}
		$W$ is $2\delta$-quasi-convex
		\item \label{enu: extended windmill - invariant}
		$W$ and $V$ are $L$-invariant.
		\item \label{enu: extended windmill - large angle}
		For every $(H,v) \in \mathcal R$ such that $v \notin V$, for every $h \in H\setminus\{1\}$, for every $x,x' \in W$, $\gro x{hx'}v \leq 100\delta$.
		\item \label{enu: extended windmill - stabilizers}
		 For every $(H,v) \in \mathcal R$ such that $v \notin V$, $\stab v \cap L = \{1\}$.
	\end{labelledenu}
\end{defi}

\rem If $N$ is the trivial group and $V$ is the set of apices contained in $W$, then we roughly recover the definition of windmill given in \cite{Dahmani:2011vu}. 
Four our purpose, $V$ may be a smaller set.
However, if an apex $v$ is not contained in $V$ then the corresponding rotation group $H$ should rotates the points of $W$ with a ``large angle''.

\begin{lemm}
\label{res: V in the 4-delta neighbourhood of W}
	If $(W,N,V)$ is an extended windmill, then $V$ is contained in the $4\delta$-neighbourhood of $W$.
\end{lemm}
\begin{proof}
	Let $(H,v)\in \mathcal{R}$, $v\in V$, and $h\in H\setminus \{1\}$.
	Let $y\in W$.
	By \autoref{res: rotation family - small product at the apex},  $\gro y{hy}{v}\leq 2 \delta$.
	By Axiom \ref{enu: extended windmill - invariant},  $W$ is $K_V$-invariant and it follows $hy\in W$. 
	By Axiom \ref{enu: extended windmill - quasi-convex}, $W$ is $2\delta$-quasi-convex subset of $X$,
	and therefore, $d(v,W)\leq \gro y{hy}{v}+2\delta \leq 4\delta$.
\end{proof}

\begin{prop}
\label{res: growing windmill}
	Let $(W,N,V)$ be an extended windmill. 
	Let $L$ be the subgroup generated by $N$ and $K_V$.
	There exists a subset $W'$ of $X$ with the following properties.
	\begin{enumerate}
	\item[(a)] The $(\sigma/10)$-neighbourhood of $W$ is contained in $W'$.
	\item[(b)] The triple $(W',N,V')$ is an extended windmill, where  $V' = W' \cap v(\mathcal R)$.
	\item[(c)] There exists a subset $\mathcal R_0$ of $\mathcal R$ such that the subgroup $L'$ generated by $N$ and $K_{V'}$ is isomorphic to the following free product
	\begin{equation*}
		L' = L * \left( *_{(H,v)\in \mathcal R_0}H\right) 
	\end{equation*}
	
	\end{enumerate}		
\end{prop}

\begin{proof}
	Let us denote by $A$ the following set of apices
	\begin{displaymath}
		A = \set{\fantomB v \in v(\mathcal R)\setminus V}{d(v,W) \leq3\sigma/10}.
	\end{displaymath}
	We consider two cases, depending on whether $A$ is empty or not.
	
	\paragraph{Case 1.} \emph{Assume that $A$ empty}.
	We choose for $W'$ the $\sigma/10$-neighborhood of $W$.
	Clearly (a) holds.
	By \autoref{res: V in the 4-delta neighbourhood of W} and since $\sigma$ is much greater than $\delta$, we get that $V'=W'\cap v(\mathcal R)=V$. 
 	To see that $(W', N, V)$ is an extended windmill, observe that \ref{enu: extended windmill - quasi-convex} follows by \autoref{res: neighborhood of a quasi-convex};
	\ref{enu: extended windmill - invariant} and \ref{enu: extended windmill - stabilizers} follows trivially since $V=V'$ then $L' = \langle N, K_{V'} \rangle = L$ and $(W,N,V)$ is an extended windmill. 
	It remains \ref{enu: extended windmill - large angle}, which follows from \autoref{res: rotation family - small product and far qconvex} and bearing in mind that $A=\emptyset$.
	Note also that (c) holds because $L = L'$.

	\paragraph{Case 2.} 
	\emph{Assume that $A$ is not empty}.
	We denote by $S$ (like \emph{sail}) the hull of $W \cup A$ (see \autoref{def: hull}).
	We are going to let this sail ``turn'' around the apices of $A$.
	Let $W'$ be the $\sigma/10$-neighborhood of $K_A \cdot S$ and $V' = W' \cap v(\mathcal R)$.
	In particular, $W'$ contains the $\sigma/10$-neighborhood of $W$, and hence (a) holds.
	The goal is to prove that (b) and (c) hold.
	The following observation will be useful: since $V$ and $W$ are both $L$-invariant, hence so is $A$ (and thus $S$) and thus
	\begin{equation}
	\label{eq: L normalizes K_A}
		W'\text{ is }\langle L,K_A\rangle\text{-invariant}, L \text{ normalizes } K_A \text{ and hence }\langle L, K_A \rangle \cdot S= K_A \cdot S.
	\end{equation}

	
	\begin{lemm}
	\label{res: windmill - gromov product at the apex}
		Let $(H,v) \in \mathcal R$ such that $v \in A$. 
		Let $x,y \in S$ and $h \in H\setminus\{1\}$.
		Then $\gro x{hy}v \leq 105\delta$.
	\end{lemm}
	
	\begin{proof}
		Recall that $S$ is the hull of $W \cup A$. 
		According to \autoref{res: gromov product and hull}, it is sufficient to prove that for all $x,y \in W\cup A$, $\gro x{hy}v \leq 102\delta$.
		Let $x,y \in W\cup A$.
		Note that if $x=v$ or $y=v$, then $\gro x{hy}v = 0$ ($h$ fixes $v$).
		Therefore we can suppose that $x$ and $y$ are distinct from $v$, and we have 3 different cases.
		\paragraph{Case 1.} \emph{Assume that $x$ and $y$ lie in $W$}.
		Recall that $v$ does not belong to $V$, thus by Axiom \ref{enu: extended windmill - large angle}, $\gro x{hy}v \leq 100\delta$.
		
		\paragraph{Case 2.} \emph{Assume that $x$ lies in $W$ and $y$ in $A-\{v\}$}.
		We denote by $r$ and $q$  $\delta$-projections of $v$ and $y$ on $W$ respectively.
		We claim that $\gro yqv> 101\delta$.
		By \autoref{res: proj quasi-convex}~(ii), 
		\begin{displaymath}
			\dist qr \leq \max \left\{ \dist yv - \dist yq - \dist vr + 14\delta, 7\delta \right\}.
		\end{displaymath}
		However $y$ and $v$ are two distinct apices. 
		It follows that $\dist yv \geq \sigma$ whereas  $\dist yq$ and $\dist vr$ are at most $3\sigma/10 + \delta$.
		The triangle inequality combined with our choice of $\sigma_0$ yields $\dist qr > 7\delta$.
		Consequently we necessarily have 
		\begin{displaymath}
			\dist yq + \dist qv \leq \dist yq + \dist qr + \dist rv \leq \dist yv + 14\delta.
		\end{displaymath}
		In particular, $\gro yvq \leq 7 \delta$.
		Hence $\gro yqv = \dist qv - \gro yvq \geq \dist qv - 7 \delta$.
		By the triangle inequality,  $\dist qv \geq \dist yv - \dist yq$.		
		Since $\dist yv  \geq \sigma$ and $\dist yq \leq 3\sigma/10 + \delta$,
		our claim follows from $\sigma \geq \sigma_0$.
		Recall that $x,q \in W$.
		It follows from Axiom~\ref{enu: extended windmill - large angle} that $\gro x{hq}v \leq 100\delta$.
		The four point inequality leads to
		\begin{equation}
		\label{eqn: windmill - gromov product at the apex}
			\min\left\{  \gro x{hy}v,  \gro {hy}{hq}v\right\} 
			\leq \gro x{hq}v + \delta
			\leq 101\delta.
		\end{equation}
		Note that $\gro yqv = \gro {hy}{hq}v$.
		Equation~\eqref{eqn: windmill - gromov product at the apex} combined with our claim gives that $\gro x{hy}v \leq 101\delta$.
		
		\paragraph{Case 3.} \emph{Assume that $x$ and $y$ lie in $A-\{v\}$}.
		Again, denote by $q$ a $\delta$-projection of $y$ on $W$. Since the lemma holds in Case 2, we get that $\gro x{hq}v \leq 101\delta$.
		Using again the four point inequality, \eqref{eqn: windmill - gromov product at the apex}, and the previous argument about $\gro{hy}{hq}v$,
		we get that $\gro x{hy}v \leq 102 \delta$.
	\end{proof}
	
	\begin{lemm}
	\label{res: windmill - apices in the neighborhood of S}
		Let $v \in v(\mathcal R)$.
		If $d(v,S) \leq \sigma/5$ then $v \in V \cup A$.
	\end{lemm}

	\begin{proof}
		Let $p$ be a $\delta$-projection of $v$ on $S$. By the definition of hull, there exists $y,y'\in W\cup A$ such that
		$p$ lies on a $(1,\delta)$-quasi-geodesic $\gamma$ with endpoint $y,y'$.
		In particular,  the triangle inequality yields to $ \gro y{y'}v \leq \gro y{y'}p + |p-v|$, and hence
		$\gro y{y'}v \leq \sigma/5 + 2\delta$.
		Let us denote by $z$ and $z'$ respective $\delta$-projections of $y$ and $y'$ on $W$.
		Applying twice the four point inequality (\ref{eqn: hyperbolicity condition 1}) we get
		\begin{equation}
		\label{eqn: windmill - apices in the neighborhood of K.S}
				\min \left\{ \gro yzv, \gro z{z'}v, \gro{y'}{z'}v \right\} \leq \gro y{y'}v +2 \delta \leq \sigma/5 + 4\delta.
		\end{equation}
		Assume first that the minimum in~(\ref{eqn: windmill - apices in the neighborhood of K.S}) is achieved by $\gro z{z'}v$.
		The windmill $W$ being $2\delta$-quasi-convex, we have 
		\begin{displaymath}
			d(v,W) \leq \gro z{z'}v + 2 \delta \leq  \sigma/5 + 6\delta \leq  3\sigma/10.
		\end{displaymath}
		By definition of $A$, $v$ is necessarily a point of $V \cup A$.		
		
		\medskip
		Assume now that the minimum is achieved by $\gro yzv$ (the proof works similarly for $\gro{y'}{z'}v$).
		It follows from the triangle inequality that $\dist yv - \dist yz \leq \gro yzv \leq \sigma/5 + 4\delta$.
		If $y\in W$, $\dist yz \leq \delta$ and $\dist yv \leq \sigma/5+5\delta$ and by definition of $A$, $v\in V\cup A$.
		If $y\in A$, by construction $\dist yz$ is bounded above by $3\sigma/10 + \delta$, thus $\dist yv < \sigma$.
		Since the distance between two distinct apices of $\mathcal R$ is at least $\sigma$, we get that  $y=v$.
		Hence $v \in A$.
	\end{proof}
	
	\begin{lemm}
	\label{res: windmill - L' invariant}
		The sets $L \cup  K_A$ and $ N\cup  K_{V'}$ generate the same subgroup $L'$ of $G$.
		Moreover $W'$ and $V'$ are $L'$-invariant.
	\end{lemm}
	
	\rem This lemma proves Axiom~\ref{enu: extended windmill - invariant} for our new windmill.
	\begin{proof}
	By construction $A\subseteq V'$.
	Since $L=\langle N, K_V \rangle$, we have that $\langle L \cup K_A \rangle \subset \langle N, K_{V'}\rangle$.
	It is enough to show that $K_{V'}\subset \langle L\cup K_A\rangle$.
	It follows from \autoref{res: windmill - apices in the neighborhood of S} that every apex contained in the $\sigma/10$-neighborhood of $K_A\cdot S$ (i.e. $W'$) actually belongs to $K_A\cdot S$. 
	Thus $V'$ is the set  $K_A \cdot (V\cup A)$. 
	Recall that $\mathcal R$ is $G$-invariant.
	The other inclusion follows.
	Moreover by \eqref{eq: L normalizes K_A}, $W'$ and $V'$ are both $L'$-invariant.
	\end{proof}

	For the remainder of the section, $L'$ denotes the subgroup $\langle L, K_A\rangle= \langle N, K_{V'}\rangle$.

	
	\paragraph{Decomposition of the elements of $L'$.}
	We denote by $\mathbf A$ a set of representatives for $A/L$.
	We use $\mathbf L$ to denote an abstract copy of $L$, and similarly for $(H,v)\in \mathcal R$, we use $(\mathbf H,\mathbf v)$ to denote an
	abstract copy of the pair. 
	We denote by $\mathbf L'$ the free product of $\mathbf L$ and the rotation groups $\mathbf H$ where $(\mathbf H,\mathbf v) \in \mathcal R$ and $\mathbf v \in \mathbf A$.
	\begin{equation*}
		\mathbf L' = \mathbf{L}* \left( *_{(\mathbf H,\mathbf v) \in \mathcal R, \mathbf v \in \mathbf A} \mathbf H\right).
	\end{equation*}
	It comes with a natural morphism $\mathbf L' \rightarrow L'$.
	By construction this map in onto.
	We are going to prove (among other things) that it is an isomorphism.
	Let $\mathbf g$ be an element of $\mathbf L'$.
	It can be written $g = \mathbf u_0 \mathbf h_1 \mathbf u_1 \dots \mathbf u_{m-1}\mathbf h_m \mathbf u_m$, where
	\begin{enumerate}
		\item for every $i \in \intvald 1m$ there exists $(\mathbf H_i,\mathbf v_i) \in \mathcal R$ with $\mathbf v_i \in \mathbf A$ such that $\mathbf h_i \in \mathbf H_i\setminus \{1\}$,
		\item for every $i \in \intvald 0m$, $\mathbf u_i \in \mathbf{L}$,
		\item for every $i \in \intvald 1{m-1}$, if $\mathbf u_i=1$ then $\mathbf v_i \neq \mathbf v_{i+1}$.
	\end{enumerate}
	The integer $m$ does not depend on the decomposition above. 
	We call it the \emph{number of rotation of $\mathbf g$} and denote it by $m(\mathbf g)$.
	The image $g$ of $\mathbf g$ in $L'$ can be rewritten as follows
	\begin{eqnarray*}
		g 
		& = & \left[\mathbf u_0\mathbf h_1\mathbf u_0^{-1}\right]\left[(\mathbf u_0\mathbf u_1)\mathbf h_2(\mathbf u_0\mathbf u_1)^{-1}\right] \dots \left[(\mathbf u_0\dots \mathbf u_{m-1})\mathbf h_m(\mathbf u_0\dots \mathbf u_{m-1})^{-1}\right] \mathbf u_0\dots \mathbf u_m \\
		& = & h_1\dots h_m u, 
	\end{eqnarray*}
	where $u = \mathbf u_0 \dots \mathbf u_m$ is in $L$ and for every $i \in \intvald 1m$, $h_i = (\mathbf u_0\dots \mathbf u_{i-1})\mathbf h_i(\mathbf u_0\dots \mathbf u_{i-1})^{-1}$ is an element of the rotation group $H_i$ fixing the vertex $v_i = \mathbf u_0\dots \mathbf u_{i-1}\mathbf v_i$.
	Since $A$ is $L$-invariant, all the apices $v_i$ belongs to $A$.
	We claim that for every $i \in \intvald i{m-1}$, $v_i \neq v_{i+1}$.
	Let $i \in \intvald 1{m-1}$.
	Assume that on the contrary our claim is false.
	Then $\mathbf v_i = \mathbf u_i\mathbf v_{i+1}$.
	The points $\mathbf v_i$ and $\mathbf v_{i+1}$ both belongs to $\mathbf A$, the set of representatives of $A/L$.
	It follows that $\mathbf v_i = \mathbf v_{i+1}$ is fixed by $\mathbf u_i$.
	According to Axiom~\ref{enu: extended windmill - stabilizers}, $\mathbf u_i$ is necessarily trivial.
	It contradicts property (iii) of the decomposition of $\mathbf g$ in $\mathbf L'$.
	
	\medskip
	The second way of writing the elements of $L'$, namely $g = h_1h_2\dots h_mu$, is shorter and will be preferred 
	and used in the next lemma.
	Note that the integer $m$ that appears in the second form is still the number of rotations of $\mathbf g$.
	
	\begin{lemm}
	\label{res: windmill - pre quasi-convex}
		Let $y, y' \in S$.
		Let $\mathbf g \in \mathbf L'$, $g$ its image in $L'$ and $m$ its number of rotations.
		There exists a sequence of points $y=y_0, \dots, y_{m+1}=gy'$ of $X$ satisfying the following properties
		\begin{enumerate}
			\item \label{enu: windmill - pre quasi-convex - points in S}
			for all $i \in \intvald 1{m+1}$ there exists $g_i \in L'$ such that $g_i^{-1}y_{i-1}$ and $g_i^{-1} y_i$ belong to $S$,
			\item \label{enu: windmill - pre quasi-convex - distance}
			for all $i \in \intvald 1{m-1}$, $\dist {y_{i+1}}{y_i} \geq \sigma$,
			\item \label{enu: windmill - pre quasi-convex - gromov produt}
			for all $i, j, k \in \intvald 1m$ with $i \leq j \leq k$ we have  $\gro {y_i}{y_k}{y_j} \leq 110\delta$,
			\item \label{enu: windmill - pre quasi-convex - quasi-geodesic}
			For all $x \in X$ there exists $i \in \intvald 0m$ such that $\gro{y_{i+1}}{y_i}x \leq \gro y{gy'}x + 218\delta$.
		\end{enumerate}
	\end{lemm}
	
	\begin{proof}
		According to our previous discussion $g$ can be written $h_1\dots h_m u$ where $u \in L$ and  for every $i \in \intvald 1m$ there exists $(H_i,v_i) \in \mathcal R$ with $v_i \in A$ such that $h_i \in H_i \setminus \{1\}$.
		Moreover two consecutive apices $v_i$ and $v_{i+1}$ are distinct.
		If $m=0$, i.e. $g$ belongs to $L$, then the points $y_0 = y$ and $y_1 = gy'$ lie in $S$ and hence satisfy the conclusion of the lemma.
		Assume now that $m \geq 1$.
		For all $i \in \intvald 1m$, we put $g_i = h_1\dots h_{i-1}$ and $y_i = g_iv_i$.
		Moreover, we put $g_{m+1}=h_1\dots h_m = gu^{-1}$, $y_0 = y$ and $y_{m+1} = gy'$.
		
		\begin{itemize}
		\item Let $i \in \intvald 1{m+1}$. 
		If $i\neq m+1$ then $g_i^{-1}y_i = v_i$. 
		Otherwise $g_{m+1}^{-1}y_{m+1} = uy'$.
		However $u$ belongs to $L$ which stabilizes $S$.
		Thus $g_i^{-1}y_i$ belongs to $S$.
		Assume now that $i \neq 1$.
		Recall that $h_{i-1}$ fixes $v_{i-1}$ thus $g_i^{-1}y_{i-1} = h_{i-1}^{-1}v_{i-1} = v_{i-1}$.
		On the other hand $g_1^{-1}y_0 = y$.
		By construction $g_i^{-1}y_{i-1}$ is a point of $S$. 
		This completes the proof of Point~\ref{enu: windmill - pre quasi-convex - points in S}.
			\item 
			Let $i \in \intvald 1{m-1}$.
			The apex $v_i$ is fixed by $h_i$ therefore
			\begin{displaymath}
				\dist {y_{i+1}}{y_i} = \dist{g_ih_iv_{i+1}}{g_iv_i} = \dist {v_{i+1}}{v_i}.
			\end{displaymath}
			However $v_{i+1}$ and $v_i$ are two distinct apices of $\mathcal R$, therefore $\dist{y_{i+1}}{y_i}\geq \sigma$.
			This proves Point~\ref{enu: windmill - pre quasi-convex - distance}.
			
			\item 
			Let $i \in \intvald 1m$. By construction $g_i^{-1} y_i=v_i$ whereas $g_i^{-1}y_{i-1}$ belongs to $S$.
			On the other hand $g_i^{-1}y_{i+1} = h_ig_{i+1}^{-1}y_{i+1}$.
			By \autoref{res: windmill - gromov product at the apex}, $\gro {y_{i-1}}{y_{i+1}}{y_i} = \gro{g_i^{-1}y_{i-1}}{g_i^{-1}y_{i+1} }{v_i}$ is bounded above by $105\delta$.
		\end{itemize}		
		
		We chose the constant $\sigma$ big enough compared to $\delta$ in such a way that we can apply \autoref{res: stability discrete quasi-geodesic} to the sequence $y_0,\dots,y_{m+1}$. 
		Point~\ref{enu: windmill - pre quasi-convex - gromov produt} and Point~\ref{enu: windmill - pre quasi-convex - quasi-geodesic} follow from the stability of discrete quasi-geodesics.
	\end{proof}
	
	\begin{lemm}
	\label{res: windmill - W' quasi-convex}
		The set $K_A\cdot S$ is $224\delta$-quasi-convex whereas $W'$ is $2\delta$-quasi-convex.
	\end{lemm}
	
	\rem This lemma proves Axiom~\ref{enu: extended windmill - quasi-convex} for our new windmill.
	Let $v$ be an apex of $v(\mathcal R)$ which is not in $V'$.
	According to \autoref{res: windmill - apices in the neighborhood of S} we have $d(v, W') \geq \sigma/10$.
	Since $W'$ is quasi-convex, it follows from \autoref{res: rotation family - small product and far qconvex} that our new windmill satisfies Axiom~\ref{enu: extended windmill - large angle}. 
	
	\begin{proof}
		The set $W'$ was defined as the $\sigma/10$-neighborhood of $K_A\cdot S$.
		According to \autoref{res: neighborhood of a quasi-convex}, it is sufficient to show that $K_A\cdot S$ is $224\delta$-quasi-convex.
		Let $x \in X$ and $y,y' \in K_A\cdot S$.
		It follows from \autoref{res: windmill - pre quasi-convex} that there exist $z,z' \in S$ and $g \in L'$ such that $\gro {gz}{gz'}x \leq \gro y{y'}x + 218\delta$.
		However $S$ being a hull, it is $6\delta$-quasi-convex (\autoref{res: hull quasi-convex}).
		In particular so is $gS$. 
		By \eqref{eq: L normalizes K_A}, we have that $gS \subseteq K_A \cdot S$.
		Therefore 
		\begin{displaymath}
			d(x,K_A\cdot S) \leq d(x,gS) \leq \gro {gz}{gz'}x +6\delta\leq \gro y{y'}x + 224\delta. \qedhere
		\end{displaymath}
	\end{proof}

	\begin{lemm}
	\label{res: windmill - pretranslation}
		Let $\mathbf g$ be an element of $\mathbf L'$ and $g$ its image in $L'$.
		One of the following holds
		\begin{enumerate}
			\item $\mathbf g$  belongs to $\mathbf L$ (and thus $g\in L$).
			\item There exists $(H,v) \in \mathcal R$, with $v \in A$ such that $g \in H \setminus \{1\}$.
			\item For every $y \in S$, $\dist {gy}y \geq \sigma -440\delta$. 
		\end{enumerate}
	\end{lemm}

	\begin{proof}
		Let $m$ be the rotation number of $\mathbf g$.
		Suppose that $m=0$.
		Then $\mathbf g$ belongs to $\mathbf{L}$, which gives the first case.
		Suppose that $m=1$.
		There exists $(H,v) \in \mathcal R$ with $v \in A$, $h \in H\setminus\{1\}$ and $u \in L$ such that $g = hu$.
		If $u=1$, then $g$ belongs to $H\setminus\{1\}$, which gives the second case.
		Therefore we can assume that $u \neq 1$.
		Since $u$ belongs to $L$, Axiom~\ref{enu: extended windmill - stabilizers} yields $uv \neq v$.
		Let $y \in S$.
		It follows from the triangle inequality that 
		\begin{equation*}
			\dist vy +\dist{uy}v = \dist{uv}{uy} + \dist{uy}v \geq \dist {uv}v \geq \sigma
		\end{equation*}
		On the other hand, both $y$ and $uy$ belong to $S$.
		By \autoref{res: windmill - gromov product at the apex}, we get $\gro{huy}yv \leq 105\delta$.
		Hence, 
		\begin{equation*}
			\dist{gy}y 
			= \dist {huy}y
			\geq \dist{huy}v + \dist vy - 210\delta 
			= \dist{uy}v + \dist vy - 210\delta 
			\geq \sigma - 210\delta,
		\end{equation*}
		which gives the third case. We have proved the lemma when $m\leq 1$.
		
		\medskip
		Suppose now that $m \geq 2$.
		Let $y \in S$.
		According to \autoref{res: windmill - pre quasi-convex}, there exists a sequence of points $y=y_0, \dots, y_{m+1} = gy$ with the following properties.
		\begin{itemize}
			\item $\gro {y_0}{y_{m+1}}{y_1} \leq 110\delta$ and $\gro {y_1}{y_{m+1}}{y_2} \leq 110\delta$.
			\item $ \dist{y_1}{y_2} \geq \sigma$.
		\end{itemize}
		In particular,
		\begin{equation*}
			\dist {gy}y  \geq \dist {y_0}{y_1} + \dist{y_1}{y_2} + \dist{y_2}{y_{m+1}} - 440\delta \geq \sigma - 440\delta. \qedhere
		\end{equation*}
	\end{proof}
	
	\begin{lemm}
	\label{res: windmill - stabilizers}
		For every $(H,v) \in \mathcal R$ with $v \notin V'$, we have $\stab v \cap L' = \{1\}$.
	\end{lemm}
	
	\rem The conclusion of the lemma corresponds to Axiom~\ref{enu: extended windmill - stabilizers} for our new windmill. 
	Recall that Axiom~\ref{enu: extended windmill - quasi-convex}, \ref{enu: extended windmill - invariant} and \ref{enu: extended windmill - large angle} have been already proved.
	It finishes the proof that the statement (b) of the proposition holds.
	
	\begin{proof}
		Let $(H,v) \in \mathcal R$ such that $v \notin V'$.
		Let $g \in L'$ such that $gv=v$.
		Let $y$ be a $\delta$-projection of $v$ on $K_A\cdot S$.
		There exists $w \in K_A$ such that $wy$ belongs to $S$.
		Recall that $V'$ is $K_A$-invariant. 
		Note that by \autoref{res: V in the 4-delta neighbourhood of W} and \autoref{res: windmill - apices in the neighborhood of S},
		we get that  $V$ is a subset of $V'$.
		Thus $wgw^{-1}$ fixes the apex $wv$ which does not belong to $V'$ and neither to $V$.
		We apply \autoref{res: windmill - pretranslation} to $wgw^{-1}$.
		We distinguish three cases.
		
		Assume first that $wgw^{-1} \in L$.
		Since $wgw^{-1}$ fixes an apex $wv\notin V$, Axiom~\ref{enu: extended windmill - stabilizers}, implies that $g$ is trivial.

		Assume now that there exists $(H',v') \in \mathcal R$ with $v' \in A$ such that $wgw^{-1} \in H' \setminus\{1\}$.
		A non trivial element of a rotation groups fixes exactly one points.
		However $wgw^{-1}$ fixes $v' \in V'$ and $wv \notin V'$.
		This case never happens.
		
		The last case states that $\dist {gy}y  = \dist{(wgw^{-1})wy}{wy}  \geq \sigma - 440\delta$.
		The points $y$ and $gy$ are respective $\delta$-projections of $v$ and $gv$ on $K_A\cdot S$, which is $224\delta$-quasi-convex.
		It follows from \autoref{res: proj quasi-convex} that 
		\begin{equation*}
			\dist {gy}y \leq \max\left\{\dist {gv}v - 2 \dist vy + 902\delta, 451\delta \right\}.
		\end{equation*}
		Since $\dist{gy}y \geq \sigma - 440\delta$ we get $\dist{gv}v \geq \sigma - 1342\delta$.
		Thus $g$ cannot fix $v$. 
		This case also never happens.
	\end{proof}

	\begin{lemm}
	\label{res: windmill - free product}
		The canonical map $\mathbf L' \rightarrow L'$ is one-to-one.
	\end{lemm}
	
	\begin{proof}
		Let $\mathbf g$ be an element of $\mathbf L'$ whose image $g$ in $L'$ is trivial.
		It follows from \autoref{res: windmill - pretranslation} that $\mathbf g$ belongs to $\mathbf{L}$.
		By construction the map $\mathbf L' \rightarrow L'$ induces an embedding of $\mathbf{L}$ into $L'$, hence $\mathbf g = 1$.
	\end{proof}
	
	\rem	\autoref{res: windmill - free product} shows that property (c) holds and concludes the proof of \autoref{res: growing windmill}.
\end{proof}

\begin{proof}[Proof of \autoref{res: advanced rotation family}]
	Let $\mathcal R$ be $\sigma$-rotation family, and let $Y$ and $N$ as in the hypothesis of the theorem.
	Note that $(Y,N,\emptyset)$ is an extended windmill, which we  denote by $(W_0, N, V_0)$.
	A proof by induction using \autoref{res: growing windmill} shows that for every $n \in \N$ there is an extended windmill $(W_n, N, V_n)$ with the following property.
	If $L_n$ stands for the subgroup generated by $N$ and $K_{V_n}$, then for every $n \in \N\setminus\{0\}$,
	\begin{enumerate}
		\item $W_n$ contains the $\sigma/10$-neighborhood of $W_{n-1}$;
		\item $V_n = W_n\cap v(\mathcal R)$;
		\item there exists a subset $\mathcal R_n$ of $\mathcal R$ such that $L_n$ is isomorphic to the free product of $L_{n-1}$ and the rotation groups of $H(\mathcal R_n)$.
	\end{enumerate}
	Note also that $L_0 = N$.
	Since the sequence of subsets $(W_n)$ is growing every vertex of $v(\mathcal R)$ ultimately belongs to some $V_n$.
	In other words $(L_n)$ is an increasing sequence of subgroups of $G$ whose union $L$ is exactly the subgroup generated by $N$ and $K$.
	Let $\mathcal S$ be the union of all $\mathcal R_n$.
	It follows from the free product structure of every $L_n$ that $L$ is isomorphic to the free product of $N$ and the rotation groups of $H(\mathcal S)$.
\end{proof}

\paragraph{Application.}
The goal of this paragraph is to prove the following statement.

\begin{theo}
\label{res: rotation family with a reduced element}
	Let $X$ be $\delta$-hyperbolic space and $G$ a group acting by isometries on it.
	There exists $\sigma_0 >0$ with the following property.
	Let $\mathcal R$ be a $\sigma$-rotation family with $\sigma >\sigma_0$.
	Let $K$ be the (normal) subgroup generated by all the rotation groups of $H(\mathcal R)$ and $\bar G$ be the quotient $G/K$
	Then the following holds.
	\begin{enumerate}
		\item The quotient $\bar X = X/K$ is $\bar \delta$-hyperbolic with $\bar \delta \leq 900 \delta$.
		\item For every $\bar g  \in \bar G$ acting loxodromically on $\bar X$, there exists a pre-image $g \in G$ of $\bar g$ and a subset $\mathcal S$ of $\mathcal R$ such that 		\begin{equation*}
			\langle g, K \rangle = \langle g \rangle * \left(*_{H \in H(\mathcal S)} H\right).
		\end{equation*}
	\end{enumerate}
\end{theo}

Let $\delta > 0$.
From now on $\sigma_0$ is the maximum of the constants respectively given by \autoref{res: standard rotation family} and \autoref{res: advanced rotation family}.
Up to increasing the value of $\sigma_0$ we can always assume that $\sigma_0 \geq L_S\delta + 150\delta$,
where $L_S$ is the constant of \autoref{def: constant LS}.
Let $X$ be a $\delta$-hyperbolic space endowed with an action by isometries of a group $G$.
Let $\mathcal R$ be $\sigma$-rotation family with $\sigma > \sigma_0$.
Recall that for $g\in G$ that is a loxodromic isometry of $G$,  $Y_g$ denotes the cylinder of $g$ (see Definition \ref{def: cylinder loxodromic element}).
\begin{defi}
\label{def: reduced element}
	Let $g$ be a loxodromic element of $G$. We say that $g$ is \emph{$\mathcal R$-reduced} if for every $(H,v) \in \mathcal R$, for every $h \in H\setminus\{1\}$, for every $y,y' \in Y_g$, $\gro {hy}{y'}v \leq 100\delta$.
\end{defi}

\begin{lemm}
\label{res: existence of reduced element}
	Let $g \in G$.
	There exists $u \in K$ such that $ug$ is either not loxodromic or $\mathcal R$-reduced.
\end{lemm}

\begin{proof}
	We assume that for every $u \in K$, $ug$ is loxodromic.
	We now fix $u_0 \in K$ such that 
	\begin{equation}
	\label{eq: def u_0}	
		\text{for every $u \in K$, $\len {u_0g} \leq \len{ug} + \delta$.}
	\end{equation}	
	For simplicity we write $f = u_0g$.
	The goal is to prove that $f$ is $\mathcal R$-reduced.
	Assume on the contrary that it is not.
	There exists $(H,v) \in \mathcal R$, $h \in H\setminus\{1\}$ and $y,y' \in Y_{f}$ such that $\gro {hy}{y'}v > 100\delta$.
	We first claim that $v$ is $5\delta$-close from $Y_{f}$.
	According to \autoref{res: rotation family - small product at the apex}, $\gro{hy}yv \leq 2\delta$.
	Using the four point inequality we have 
	\begin{equation*}
		\min\left\{ \gro{hy}{y'}v , \gro {y'}yv \right\} \leq \gro {hy}yv + \delta \leq 3\delta.
	\end{equation*}
	By assumption the minimum cannot be achieved by $ \gro y{hy'}v$ thus $\gro{y'}yv \leq 3\delta$.
	However $Y_{f}$ is $2\delta$-quasi-convex, thus $v$ is $5\delta$-close from $Y_{f}$.
	In particular $\len {f} \geq \dist{fv}v - 122\delta$ (\autoref{res: moving point of a cylinder}).
	Since $f$ is loxodromic, it cannot fix $v$. 
	It follows from \ref{enu: rotation family - apices apart} that $\len {f} \geq \sigma - 122\delta \geq L_S\delta$.
	We fix $\gamma \colon \R \rightarrow X$ a $\delta$-nerve of $f$.
	Note that $Y_f$ is contained in the $27\delta$-neigborhood of $\gamma$ (see \autoref{def: nerve} and the discussion afterwards).
	\begin{figure}[htbp]
		\centering
		\includegraphics[page=1]{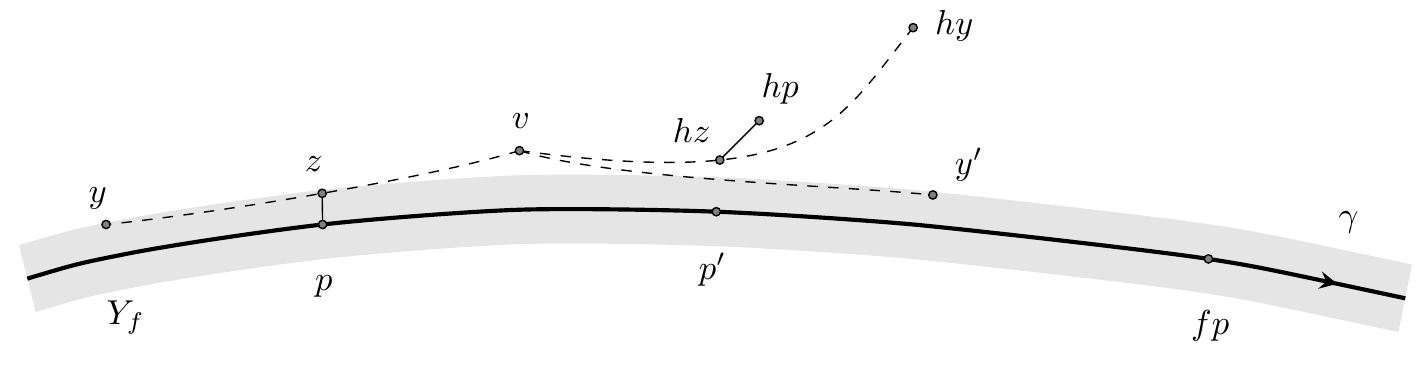}
		\caption{Shortening the translation length of $f$.}
		\label{fig: existence of reduced elements}
	\end{figure}
	Let $z$ be a point of $X$ such that $\gro yvz \leq \delta$ and $\dist vz = 100 \delta$ and $p$ a projection of $z$ on $\gamma$ (see \autoref{fig: existence of reduced elements}).
	Note that $\dist vz < \max \{ \sigma/10, \gro {hy}{y'}v\}$.
	The points $y$ and $v$ both belong to the $5\delta$-neighborhood of $Y_f$ which is $2\delta$-quasi-convex (\autoref{res: neighborhood of a quasi-convex}), hence $d(z,Y_f) \leq 8 \delta$.
	On the other hand, the cylinder $Y_{f}$ is contained in the $27\delta$-neighborhood of $\gamma$, thus $\dist zp \leq 35\delta$.
	
	We want to use the previous reasoning to bound $\dist{hz}{p'}$ where $p'$ is a $\delta$-projection of $hz$.	
	For that, we need to bound $\gro{y'}v{hz}$. This follows as a consequence of (\ref{res: metric inequalities - thin triangle}) applied to $v,z,hz$ and $y'$ and bearing in mind that $\dist{v}{z}< \gro{y'}{hy}v$. Indeed,  we get 
	\begin{eqnarray*}
		\gro {y'}v{hz} & \leq \max  \{ \dist{v}{hz} - \gro{y'}{hy}v, \gro{v}{hy}{hz}\}+ \delta\\
					   &	= \max  \{ \dist{v}{z}- \gro{y'}{hy}v, \gro{v}{y}{z} \} +\delta = 2\delta.
	\end{eqnarray*}
	Now, reasoning as previously we get $\dist{hz}{p'} \leq 36\delta$.
	Combined with \ref{enu: rotation family - large angle} it yields
	\begin{displaymath}
		\dist p{p'} \geq \dist{hz}z - 71 \delta \geq 2 \dist vz - 71\delta \geq 129 \delta.
	\end{displaymath}
	Up to changing $f$ by its inverse we can always assume that $fp$ and $p'$ are in the same connected component of $\gamma\setminus\{p\}$.
	Recall that $\gamma$ is an $\len {f}$-local $(1,\delta)$-quasi-geodesic.
	Thus $\dist{fp}{p'} \leq \dist{fp}p - \dist p{p'} + \delta$.
	However $p$ being a point of $\gamma$, $\dist{fp}p \leq \len{f} + \delta$.
	On the other hand $\dist{hp}{p'} \leq \dist{hp}{hz} + \dist{hz}{p'} \leq 71\delta$.
	Consequently 
	\begin{equation*}
		\len{h^{-1}u_0g} = \len{h^{-1}f} \leq \dist{fp}{hp} \leq \dist{fp}{p'} + \dist{hp}{p'} \leq \len {f} - 56\delta < \len{u_0g} - \delta.
	\end{equation*}
	Since $h^{-1}u_0\in K$, this last inequality contradicts  \eqref{eq: def u_0}.
	Hence $f = u_0g$ is $\mathcal R$-reduced.
\end{proof}

\begin{proof}[Proof of \autoref{res: rotation family with a reduced element}]
According to \autoref{res: standard rotation family} the quotient space $\bar X = X/K$ is $\bar \delta$-hyperbolic with $\bar \delta \leq 900 \delta$.
Let $\bar g$ be an element of $\bar G$, loxodromic for its action on $\bar X$.
In particular, $\len[stable]{\bar g} >0$.
By construction the projection $X \twoheadrightarrow \bar X$ is $1$-Lipschitz.
Thus every pre-image of $\bar g$ is loxodromic.
According to \autoref{res: existence of reduced element}, there exists a preimage $g$ of $\bar g$ which is $\mathcal R$-reduced.
Let $Y$ be the cylinder of $g$ and $N$ the cyclic group generated by $g$.
We need to check that the hypothesis of \autoref{res: advanced rotation family} hold.
First, by \autoref{res: cylinder strongly quasi-convex}, $Y$ is $2\delta$-quasi-convex.
By definition $Y$ is $N$-invariant.
Since $g$ is $\mathcal R$-reduced, condition (i) of  \autoref{res: advanced rotation family} hold.
Finally, as $g$ is loxodromic, $g$ cannot fix a point, thus for every $(H,v) \in \mathcal R$, we get $\stab v \cap N = \{1\}$.
\end{proof}

%
\section{Relatively hyperbolic groups and rotation families}
%
\label{sec : relative hyperbolicity}

There exists several definitions for the concept of relatively hyperbolic groups,
we present here to the one of Osin. We refer to \cite{Osin:2006cf} for details
and the equivalence with the concepts of relative hyperbolicity of Bowditch and strongly relative hyperbolicity in the sense of Farb.
Also see \cite{Hruska:2010iw} for a definition of relative hyperbolicity that mimics the one of geometrically finite hyperbolic groups.

\begin{defi}
\label{def: relatively hyperbolic}
	A group finitely generated group $G$ is \emph{ hyperbolic relative to a family of subgroups $\mathcal P=\{P_1,\dots, P_n\}$}, 
	if it admits a finite presentation relative to $\mathcal P$
	and this presentation has a linear relative Dehn function. 
	The subgroups $P_1,\dots, P_n$ are called  {\it peripheral (or parabolic) subgroups } of $G.$
\end{defi}

\rem We have chosen to state the present definition for sake of conciseness. 
It is worth noticing that  \autoref{def: relatively hyperbolic} in the case that $\mathcal P$ is empty,
recovers the definition of an hyperbolic group in terms of isoperimetric inequalities.

\medskip
The connection between relatively hyperbolicity and rotation families follows from
\cite[Proposition 7.7 and Corollary 7.8.]{Dahmani:2011vu}, a version of which, is stated below.

\begin{prop}
	\label{res: rotating family for relatively hyperbolic}
	Let $G$ be a finitely generated group hyperbolic relative to a finite family of subgroups $\{P_1,\dots, P_n\}$.
	There exists $\sigma_0>0$ and $\delta>0$ such that for every $\sigma\geq \sigma_0$
	there is a $\delta$-hyperbolic space $X$,
	and a finite subset $\Phi$ of $G\setminus\{1\}$ with the following properties.
	For every collection of normal subgroups  $N_i \trianglelefteq P_i$ with $N_i\cap \Phi=\emptyset$,
	$i=1,2,\dots, n$ there is a $\sigma$-rotating family $\mathcal{R}$ where for every $(H,v)\in \mathcal R$, $H$ is conjugate to some $N_i$.
	Moreover, if each $N_i$ is of finite index in  $P_i$, and $K$ is the normal subgroup of $G$ generated by $N_1\cup \dots \cup N_n$, 
	then $\bar G \coloneqq G/K$ acts properly discontinuously and cocompactly on $\bar X\coloneqq X/K$ which is $\bar\delta$-hyperbolic for some
	$\bar \delta\leq 900 \delta$. In particular, $\bar G$ is an hyperbolic group. 
\end{prop}

\rem For sake of conciseness, we have preferred to state this  version of \cite[Proposition 7.7. and Corollary 7.8]{Dahmani:2011vu} because it is suitable for our applications.  
Roughly, to see how \autoref{res: rotating family for relatively hyperbolic}
follows from \cite[Proposition 7.7.]{Dahmani:2011vu}, one should start with the horoball
definition of relatively hyperbolic group used in  \cite[Definition 7.1.]{Dahmani:2011vu}
rather than \autoref{def: relatively hyperbolic}. From that, one uses the construction
of  \cite[Lemma 7.2.]{Dahmani:2011vu} which roughly replaces the each horoball $Y$ by a cone $Y \times [0, \sigma]/\sim$
(where $\sim$ identifies all the points of the form $(y,0)$), to produced a
$\delta$-hyperbolic space. It is important to note that $\delta$ depends on $\sigma_0$
but not on $\sigma$. The apices of these cones will be the apices of the rotation family
and are stabilized by conjugates of the parabolic subgroups.
The first part of the proposition is now implied by \cite[Proposition 7.7.]{Dahmani:2011vu}
It is worth noticing that the reason of the condition of  avoiding a finite set $\Phi$
is to guarantee that the rotation groups ``rotate with a large angle''. 
The moreover part follows from  \cite[Corollary 7.8]{Dahmani:2011vu} and the particular
estimate for $\bar \delta$ follows from \autoref{res: standard rotation family}.
An interested reader can check the details in   
\cite[\S 7.1.]{Dahmani:2011vu}.

\medskip
We now  obtain an extension  (Item (iii) of theorem below) of the Dehn fillings result for relatively hyperbolic groups, which was proved originally by Osin \cite{Osin:2007ia} and 
Groves and Manning \cite{Groves:2008ip}. 

\begin{theo}
	\label{res: Dehn filling with a reduced element}
	Let $G$ be finitely generated group hyperbolic relatively to a family of subgroups $\{P_1,\dots, P_n\}$.
	There is a finite set $\Phi \subseteq G\setminus\{1\}$ such that
	whenever we take finite index normal subgroups $N_i\trianglelefteq P_i$  with $N_i\cap \Phi=\emptyset$ for $i=1,\dots, n$ then the following hold:
	\begin{enumerate}
		\item \label{enu: Dehn filling with a reduced element - hyperbolic quotient}
		$\bar G \coloneqq G/K$ is an hyperbolic group where $K$ is the normal subgroup of $G$ generated by $N_1\cup \dots \cup N_n$.
		\item \label{enu: Dehn filling with a reduced element - free kernel}
		there exists subsets $T_i$ of $G$ for $i=1,\dots, n$ such that $K$ is isomorphic to
		$*_{i=1}^n\left(*_{t\in T_i} N_i^t\right)$.
		\item \label{enu: Dehn filling with a reduced element - free primages}
		for every $\bar g\in \bar G$ of infinite order, there is a pre-image $g$ of $\bar g$ under the natural map $G\to \bar G$ and  subsets $T_i'$ of $G$ for $i=1,\dots, n$ such that
		\begin{displaymath}
			\langle g, K \rangle = \langle g \rangle * \left[*_{i=1}^n \left(*_{t\in T_i'}N_i^t\right) \right].
		\end{displaymath}
	\end{enumerate}
\end{theo}

\begin{proof}
It follows combining \autoref{res: rotating family for relatively hyperbolic}, \autoref{res: advanced rotation family} and \autoref{res: rotation family with a reduced element}.
\end{proof}

%
\section{Farrell-Jones via Dehn fillings}
%

\label{sec : farrell jones}

\paragraph{Clousure properties of Farrell-Jones groups }

Let $\mathfrak{C}$ be the class of groups satisfying  the K- and L-theoretic Farrell-Jones Conjecture with finite wreath products (with coefficients in additive categories) with respect to the family of virtually cyclic subgroups. The statement  of the Farrell-Jones conjecture and its applications can be found in \cite{Bartels:2014jv,Bartels:2008bn}. 
We collect now some properties of the class $\mathfrak{C}$.

\begin{prop}[\cite{Bartels:2008bn, Bartels:2014ga,Bartels:2014jv,Bartels:2012ca}{\cite[Proposition~4.1]{Gandini:2013ds}}]\label{res: FJinh} The following properties hold:
	\begin{enumerate}
		\item \label{enu: FJinh - subgroups}
		$\mathfrak{C}$ is closed under taking subgroups.
		\item \label{enu: FJinh - free products}
		 $\mathfrak{C}$ is closed under free products.
		\item \label{enu: FJinh - hyperbolic abelian}
		 hyperbolic groups and abelian groups are in $\mathfrak{C}$.
		\item \label{enu: FJinh - short exact sequence}
		 If $\pi\colon G\to \bar G$ is a morphism such that $\bar G$ is in $\mathfrak{C}$
		and for every torsion-free cyclic subgroup $\bar H$ of $\bar G$, $\pi^{-1}(\bar H)$ is in $\mathfrak{C}$ then $G$ is in $\mathfrak{C}$.
	\end{enumerate}
\end{prop}

\begin{theo}\label{res : farrell jones for relatively hyperbolic}
	Let $G$ be a finitely generated group hyperbolic relative to residually finite groups in the class $\mathfrak C$.
	Then $G$ is in the class $\mathfrak C$.
\end{theo}

\begin{proof}
	Let $\mathcal P = \{P_1, \dots, P_n\}$ be the peripheral subgroups of $G$.
	Let $\Phi$ be the finite subset given by \autoref{res: Dehn filling with a reduced element}.
	Recall that peripheral subgroups are residually finite.
	Hence for every $i \in \intvald 1n$, there exists a finite index normal subgroup $N_i$ of $P_i$ such that $N_i \cap \Phi = \emptyset$.
	Note that every $N_i$ belongs to $\mathfrak C$ (\autoref{res: FJinh}~\ref{enu: FJinh - subgroups}).
	Let $K$ be the normal subgroup of $G$ generated by $N_1\cup \dots \cup N_n$.
	Applying \autoref{res: Dehn filling with a reduced element} we get the following.
	\begin{enumerate}
		\item $\bar G \coloneqq G/K$ is an hyperbolic group.
		In particular it belongs to $\mathfrak C$ (\autoref{res: FJinh}~\ref{enu: FJinh - hyperbolic abelian}).
		\item There exists subsets $T_i$ of $G$ for $i=1,\dots, n$ such that $K$ is isomorphic to
		$*_{i=1}^n\left(*_{t\in T_i} N_i^t\right)$.
		As we noticed before every $N_i$ is in $\mathfrak C$.
		According to \autoref{res: FJinh}~\ref{enu: FJinh - free products} so is $K$.
		\item For every $\bar g\in \bar G$ of infinite order, there is a pre-image $g$ of $\bar g$ under the natural map $\pi \colon G\to \bar G$ and  subsets $T_i'$ of $G$ for $i=1,\dots, n$ such that
		\begin{displaymath}
			\langle g, K \rangle = \langle g \rangle * \left[*_{i=1}^n \left(*_{t\in T_i'}N_i^t\right) \right].
		\end{displaymath}
		Recall that cyclic groups are in $\mathfrak C$.
		Applying again \autoref{res: FJinh}~\ref{enu: FJinh - free products}, we get that $\langle g, K \rangle$ is in $\mathfrak C$ as well.
	\end{enumerate}
	We apply \autoref{res: FJinh}~\ref{enu: FJinh - short exact sequence} with $\pi \colon G \rightarrow \bar G$.
	Let $\bar H$ be an torsion-free cyclic subgroup of $\bar G$.
	If $\bar H$ is trivial, then we noticed that $K = \pi^{-1}(\bar H)$ is in $\mathfrak C$.
	Otherwise, there exists a loxodromic element of $\bar G$ generating $\bar H$.
	Then we observed just above that $K = \pi^{-1}(\bar H)$ in a free product that lie again in $\mathfrak C$.
\end{proof}

\todos  
\makebiblio

\end{document}